\documentclass{amsart}
\usepackage{amssymb,amsmath,amsthm}
\usepackage[dvips]{graphicx}

\let\oldmarginpar\marginpar
\renewcommand\marginpar[1]{\-\oldmarginpar[\raggedleft\footnotesize #1]%
{\raggedright\footnotesize #1}}

\begin{document}

\newtheorem{theorem}{Theorem}[section]
\newtheorem{corollary}[theorem]{Corollary}
\newtheorem{lemma}[theorem]{Lemma}
\newtheorem{proposition}[theorem]{Proposition}
\theoremstyle{definition}
\newtheorem{definition}[theorem]{Definition}
\theoremstyle{remark}
\newtheorem{remark}[theorem]{Remark}
\theoremstyle{definition}
\newtheorem{example}[theorem]{Example}

\def\rank{{\text{rank}\,}}

\numberwithin{equation}{section}

\title[Almost h-conformal semi-invariant submersions]{Almost h-conformal semi-invariant submersions from almost quaternionic Hermitian manifolds}

\author{Kwang-Soon Park}
\address{Division of General Mathematics, Room 4-107, Changgong Hall, University of Seoul, Seoul 02504, Republic of Korea}
\email{parkksn@gmail.com}



\begin{abstract}
As a generalization of Riemannian submersions, horizontally
conformal submersions, semi-invariant submersions, h-semi-invariant
submersions, almost h-semi-invariant submersions, conformal
semi-invariant submersions, we introduce h-conformal semi-invariant
submersions and almost h-conformal semi-invariant submersions from
almost quaternionic Hermitian manifolds onto Riemannian manifolds.

We study their properties: the geometry of foliations, the
conditions for total manifolds to be locally product manifolds, the
conditions for such maps to be totally geodesic, etc. Finally, we
give some examples of such maps.
\end{abstract}

\maketitle
\section{Introduction}\label{intro}
\addcontentsline{toc}{section}{Introduction}

As we know, Riemannian submersions were independently introduced by
B. O'Neill \cite{O} and A. Gray \cite{G} in 1960s. Using the notion
of almost Hermitian submersions, B. Watson \cite{W} obtained some
differential geometric properties among fibers, base manifolds, and
total manifolds. After that, many geometers study this area and
there are a lot of results on this topic.

As a generalization of  Riemannian submersions, a horizontally
conformal submersion was introduced independently by B. Fuglede
\cite{F} and T. Ishihara \cite{I} in 1970s and it is a particular
type of conformal maps.

Given a $C^{\infty}$-submersion $F$ from a Riemannian manifold
$(M,g_M)$ onto a Riemannian manifold $(N,g_N)$, according to the
conditions on the map $F : (M,g_M) \mapsto (N,g_N)$, we have the
following types of submersions:

a Riemannian submersion (\cite{G}, \cite{O}, \cite{FIP}), an almost
Hermitian submersion \cite{W}, an invariant submersion \cite{S3}, an
anti-invariant submersion \cite{S0},  a slant submersion (\cite{C},
\cite{S}), a semi-invariant submersion \cite{S2}, a semi-slant
submersion \cite{PP}, a quaternionic submersion \cite{IMV}, an
h-anti-invariant submersion and an almost h-anti-invariant
submersion \cite{P3}, an h-semi-invariant submersion and an almost
h-semi-invariant submersion \cite{P2}, a horizontally conformal
submersion (\cite{G2}, \cite{BW}), a conformal anti-invariant
submersion \cite{AS}, a conformal semi-invariant submersion
\cite{AS2}, etc.

It is well-known that Riemannian submersions are related with
physics and have their applications in the Yang-Mills theory
(\cite{BL2}, \cite{W2}), Kaluza-Klein theory (\cite{BL}, \cite{IV}),
Supergravity and superstring theories (\cite{IV2}, \cite{M}), etc.
And the quaternionic K\"{a}hler manifolds have applications in
physics as the target spaces for nonlinear $\sigma-$models with
supersymmetry \cite{CMMS}.

The paper is organized as follows. In section 2 we remind some
notions, which are needed in the following sections. In section 3 we
give the definitions of h-conformal semi-invariant submersions and
almost h-conformal semi-invariant submersions and obtain some
properties on them: the characterizations of such maps, the
harmonicity of such maps, the conditions for such maps to be totally
geodesic,
 the integrability of distributions, the geometry of foliations, etc.
 In section 4 we give some examples of
h-conformal semi-invariant submersions and almost h-conformal
semi-invariant submersions.

\section{Preliminaries}\label{prelim}

In this section we remind some notions, which will be used in the
following sections.

Let $(M,g_M)$ and $(N,g_N)$ be Riemannian manifolds, where $g_M$ and
$g_N$ are Riemannian metrics on $C^{\infty}$-manifolds $M$ and $N$,
respectively.

Let $F : (M,g_M) \mapsto (N,g_N)$ be a $C^{\infty}$-map.

We call the map $F$ a {\em $C^{\infty}$-submersion} if $F$ is
surjective and the differential $(F_*)_p$  has maximal rank for any
$p\in M$.

Then the map $F$ is said to be a {\em Riemannian submersion}
(\cite{O}, \cite{FIP}) if $F$ is a $C^{\infty}$-submersion and
$$
(F_*)_p : ((\ker (F_*)_p)^{\perp}, (g_M)_p) \mapsto (T_{F(p)} N,
(g_N)_{F(p)})
$$
is a linear isometry for any $p\in M$, where $(\ker
(F_*)_p)^{\perp}$ is the orthogonal complement of the space $\ker
(F_*)_p$ in the tangent space $T_p M$ to $M$ at $p$.

The map $F$ is called {\em horizontally weakly conformal} at $p\in
M$ if it satisfies either (i) $(F_*)_p = 0$ or (ii) $(F_*)_p$ is
surjective and there exists a positive number $\lambda (p) > 0$ such
that
\begin{equation}\label{eq1}
g_N ((F_*)_p X, (F_*)_p Y) = \lambda^2 g_M (X, Y) \quad \text{for} \
X,Y\in (\ker (F_*)_p)^{\perp}.
\end{equation}
We call the point $p$ a {\em critical point} if it satisfies the
type (i) and call the point $p$ a {\em regular point} if it
satisfies the type (ii). And the positive number $\lambda (p)$ is
said to be {\em dilation} of $F$ at $p$. The map $F$ is called {\em
horizontally weakly conformal} if it is horizontally weakly
conformal at any point of $M$. If the map $F$ is horizontally weakly
conformal and it has no critical points, then we call the map $F$ a
{\em horizontally conformal submersion}.

Let $F : (M,g_M) \mapsto (N,g_N)$ be a horizontally conformal
submersion.

Given any vector field $U\in \Gamma(TM)$, we write
\begin{equation}\label{eq2}
U = \mathcal{V}U+\mathcal{H}U,
\end{equation}
where $\mathcal{V}U\in \Gamma(\ker F_*)$ and $\mathcal{H}U\in
\Gamma((\ker F_*)^{\perp})$.

Define the (O'Neill) tensors $\mathcal{T}$ and $\mathcal{A}$ by
\begin{eqnarray}
  \mathcal{A}_E F &=& \mathcal{H}\nabla_{\mathcal{H}E} \mathcal{V}F+\mathcal{V}\nabla_{\mathcal{H}E} \mathcal{H}F  \label{eq3} \\
   \mathcal{T}_E F &=& \mathcal{H}\nabla_{\mathcal{V}E} \mathcal{V}F+\mathcal{V}\nabla_{\mathcal{V}E}
   \mathcal{H}F  \label{eq4}
\end{eqnarray}
for vector fields $E, F\in \Gamma(TM)$, where $\nabla$ is the
Levi-Civita connection of $g_M$ (\cite{O}, \cite{FIP}). Then it is
well-known that
\begin{equation}\label{eqn22}
g_M (\mathcal{T}_U V, W) = -g_M (V, \mathcal{T}_U W)
\end{equation}
\begin{equation}\label{eqn23}
g_M (\mathcal{A}_U V, W) = -g_M (V, \mathcal{A}_U W)
\end{equation}
for $U,V,W\in \Gamma(TM)$.

Define $\widehat{\nabla}_X Y := \mathcal{V}\nabla_X Y$ for $X,Y\in
\Gamma(\ker F_*)$.

Let $F : (M, g_M) \mapsto (N, g_N)$ be a $C^{\infty}$-map.

Then the {\em second fundamental form} of $F$ is given by
$$
(\nabla F_*)(X,Y) := \nabla^F _X F_* Y-F_* (\nabla _XY) \quad
\text{for} \ X,Y\in \Gamma(TM),
$$
where $\nabla^F$ is the pullback connection and we denote
conveniently by $\nabla$ the Levi-Civita connections of the metrics
$g_M$ and $g_N$ \cite{BW}.

Recall that $F$ is said to be {\em harmonic} if the tension field
$\tau(F) = trace (\nabla F_*)=0$ and $F$ is called a {\em totally
geodesic} map if $(\nabla F_*)(X,Y)=0$ for $X,Y\in \Gamma (TM)$
\cite{BW}.

\begin{lemma}\label{lem1} \cite{U}
Let $(M, g_M)$ and $(N, g_N)$ be Riemannian manifolds and $F :
(M,g_M)\mapsto (N,g_N)$ a $C^{\infty}$-map. Then we have
\begin{equation}\label{eq5}
\nabla_X^F F_*Y - \nabla_Y^F F_*X - F_*([X,Y]) = 0
\end{equation}
for $X,Y\in \Gamma(TM)$.
\end{lemma}

\begin{remark}
(1) By (\ref{eq5}), we see that the second fundamental form $\nabla
F_*$ is symmetric.

(2) By (\ref{eq5}), we obtain
\begin{equation}\label{eqn5}
[V,X]\in \Gamma(\ker F_*)
\end{equation}
for $V\in \Gamma(\ker F_*)$ and $X\in \Gamma((\ker F_*)^{\perp})$.
\end{remark}

Let $F : (M,g_M)\mapsto (N,g_N)$ be a horizontally conformal
submersion with dilation $\lambda$.

We call a vector field $X\in \Gamma(TM)$ {\em basic} if (i) $X\in
\Gamma((\ker F_*)^{\perp})$ and (ii) $X$ is $F$-related with some
vector field $\overline{X}\in \Gamma(TN)$. (i.e., $(F_*)_p X(p) =
\overline{X}(F(p))$ for any $p\in M$.)

Given any fiber $F^{-1}(y)$, $y\in N$, and any basic vector fields
$X,Y\in \Gamma((\ker F_*)^{\perp})$, we have
$$
\lambda(x)^2 g_M (X, Y)(x) = g_N (F_*X, F_*Y)(y) = constant
$$
for any $x\in F^{-1}(y)$ so that
\begin{equation}\label{eq6}
V(\lambda^2 g_M (X, Y)) =  V(g_N (F_*X, F_*Y)) = 0 \quad \text{for}
\ V\in \Gamma(\ker F_*).
\end{equation}
Then we get

\begin{proposition}\label{prop1}\cite{G2}
Let  $F : (M, g_M) \mapsto (N, g_N)$ be a horizontally conformal
submersion with dilation $\lambda$. Then we obtain
\begin{equation}\label{eq7}
\mathcal{A}_X Y = \frac{1}{2}\{\mathcal{V}[X,Y] - \lambda^2 g_M (X,
Y) \nabla_{\mathcal{V}} (\frac{1}{\lambda^2})\}
\end{equation}
for $X,Y\in \Gamma((\ker F_*)^{\perp})$.
\end{proposition}
Here, $\nabla_{\mathcal{V}}$ denotes the gradient vector field in
the distribution $\ker F_* \subset TM$. (i.e.,
$\displaystyle{\nabla_{\mathcal{V}} f = \sum_{i=1}^{m} V_i(f) V_i}$
for $f\in C^{\infty}(M)$ and a local orthonormal frame $\{ V_1,
\cdots, V_m \}$ of $\ker F_*$.)

\begin{lemma}\label{lem1}\cite{BW}
Let  $F : (M, g_M) \mapsto (N, g_N)$ be a horizontally conformal
submersion with dilation $\lambda$. Then we have
\begin{eqnarray}
 (\nabla F_*)(X,Y) &=& X(\ln \lambda) F_*Y + Y(\ln \lambda) F_*X - g_M (X, Y) F_* (\nabla \ln \lambda),   \label{eqn77} \\
 (\nabla F_*)(V,W) &=& -F_*(\mathcal{T}_V W),  \label{eqn78} \\
 (\nabla F_*)(X,V) &=& -F_*(\nabla_X V) = -F_*(\mathcal{A}_X V)  \label{eqn79}
\end{eqnarray}
for $X,Y\in \Gamma((\ker F_*)^{\perp})$ and $V,W\in \Gamma(\ker
F_*)$.
\end{lemma}

We recall some notions, which are related with our notions.

Let $(M, g_M, J)$ be an almost Hermitian manifold, where $J$ is an
almost complex structure on $M$. (i.e., $J^2 = -id$, $g_M (JX, JY) =
g_M (X, Y)$ for $X,Y\in \Gamma(TM)$.)

We call a horizontally conformal submersion $F : (M,g_M,J)\mapsto
(N,g_N)$ a {\em conformal anti-invariant submersion} \cite{AS} if
$J(\ker F_*) \subset (\ker F_*)^{\perp}$.

A horizontally conformal submersion $F : (M,g_M,J)\mapsto (N,g_N)$
is called a {\em conformal semi-invariant submersion} \cite{AS2} if
there is a distribution $\mathcal{D}_1\subset \ker F_*$ such that
$$
\ker F_* =\mathcal{D}_1\oplus \mathcal{D}_2, \
J(\mathcal{D}_1)=\mathcal{D}_1, \ J(\mathcal{D}_2) \subset (\ker
F_*)^{\perp},
$$
where $\mathcal{D}_2$ is the orthogonal complement of
$\mathcal{D}_1$ in $\ker F_*$

Let $M$ be a $4m-$dimensional $C^{\infty}$-manifold and let $E$ be a
rank 3 subbundle of $\text{End} (TM)$ such that for any point $p\in
M$ with a neighborhood $U$, there exists a local basis $\{
J_1,J_2,J_3 \}$ of sections of $E$ on $U$ satisfying for all
$\alpha\in \{ 1,2,3 \}$
$$
J_{\alpha}^2=-id, \quad
J_{\alpha}J_{\alpha+1}=-J_{\alpha+1}J_{\alpha}=J_{\alpha+2},
$$
where the indices are taken from $\{ 1,2,3 \}$ modulo 3.

Then we call $E$ an {\em almost quaternionic structure} on $M$ and
$(M,E)$ an {\em almost quaternionic manifold} \cite{AM}.

Moreover, let $g$ be a Riemannian metric on $M$ such that for any
point $p\in M$ with a neighborhood $U$, there exists a local basis
$\{ J_1,J_2,J_3 \}$ of sections of $E$ on $U$ satisfying for all
$\alpha\in \{ 1,2,3 \}$
\begin{equation}\label{hypercom}
J_{\alpha}^2=-id, \quad
J_{\alpha}J_{\alpha+1}=-J_{\alpha+1}J_{\alpha}=J_{\alpha+2},
\end{equation}
\begin{equation}\label{hypermet}
g(J_{\alpha}X, J_{\alpha}Y)=g(X, Y)
\end{equation}
for all vector fields  $X, Y\in \Gamma(TM)$, where the indices are
taken from $\{ 1,2,3 \}$ modulo 3.

Then we call $(M,E,g)$ an {\em almost quaternionic Hermitian
manifold} \cite{IMV}.

For convenience, the above basis $\{ J_1,J_2,J_3 \}$ satisfying
(\ref{hypercom}) and (\ref{hypermet}) is said to be a {\em
quaternionic Hermitian basis}.

Let $(M,E,g)$ be an almost quaternionic Hermitian manifold.

We call $(M,E,g)$ a {\em quaternionic K\"{a}hler manifold} if there
exist locally defined 1-forms $\omega_1, \omega_2, \omega_3$ such
that for $\alpha \in \{ 1,2,3 \}$
$$
\nabla_X J_{\alpha} =
\omega_{\alpha+2}(X)J_{\alpha+1}-\omega_{\alpha+1}(X)J_{\alpha+2}
$$
for any vector field $X\in \Gamma(TM)$, where the indices are taken
from $\{ 1,2,3 \}$ modulo 3 \cite{IMV}.

If there exists a global parallel quaternionic Hermitian basis $\{
J_1,J_2,J_3 \}$ of sections of $E$ on $M$ (i.e., $\nabla J_{\alpha}
= 0$ for $\alpha \in \{ 1,2,3 \}$, where $\nabla$ is the Levi-Civita
connection of the metric $g$), then $(M, E, g )$ is said to be a
{\em hyperk\"{a}hler manifold}. Furthermore, we call $(J_1, J_2,
J_3, g )$ a {\em hyperk\"{a}hler structure} on $M$ and $g$ a {\em
hyperk\"{a}hler metric} \cite{B}.

Let $(M, E_M, g_M)$ and $(N, E_N, g_N)$ be almost quaternionic
Hermitian manifolds.

A map $F : M \mapsto N$ is called a {\em $(E_M,E_N)-$holomorphic
map} if given a point $x\in M$, for any $J\in (E_M)_x$ there exists
$J'\in (E_N)_{F(x)}$ such that
$$
F_*\circ J=J'\circ F_*.
$$
A Riemannian submersion $F : M \mapsto N$ which is a
$(E_M,E_N)-$holomorphic map is called a {\em quaternionic
submersion} \cite{IMV}.

Moreover, if $(M, E_M, g_M)$ is a quaternionic K\"{a}hler manifold
(or a hyperk\"{a}hler manifold), then we say that $F$ is a {\em
quaternionic K\"{a}hler submersion} (or a {\em hyperk\"{a}hler
submersion}) \cite{IMV}.

Then it is well-known that any quaternionic K\"{a}hler submersion is
a harmonic map \cite{IMV}.

Let $(M, E, g_M)$ be an almost quaternionic Hermitian manifold and
$(N, g_N)$ a Riemannian manifold.

A Riemannian submersion $F : (M,E,g_M) \mapsto (N,g_N)$ is called an
{\em h-semi-invariant
 submersion} if given a point $p\in M$ with a
neighborhood $U$, there exists a  quaternionic Hermitian basis $\{
I,J,K \}$ of sections of $E$ on $U$ such that for any $R\in \{ I,J,K
\}$, there is a distribution $\mathcal{D}_1 \subset \ker F_*$ on $U$
such that
$$
\ker F_* =\mathcal{D}_1\oplus \mathcal{D}_2, \
R(\mathcal{D}_1)=\mathcal{D}_1, \ R(\mathcal{D}_2)\subset (\ker
F_*)^{\perp},
$$
where $\mathcal{D}_2$ is the orthogonal complement of
$\mathcal{D}_1$ in $\ker F_*$ \cite{P2}.

We call such a basis $\{ I,J,K \}$ an {\em h-semi-invariant basis}.

A Riemannian submersion $F : (M,E,g_M) \mapsto (N,g_N)$ is called an
{\em almost h-semi-invariant submersion} if given a point $p\in M$
with a neighborhood $U$, there exists a  quaternionic Hermitian
basis $\{ I,J,K \}$ of sections of $E$ on $U$ such that for each
$R\in \{ I,J,K \}$, there is a distribution $\mathcal{D}_1^R \subset
\ker F_*$ on $U$ such that
$$
\ker F_* =\mathcal{D}_1^R\oplus \mathcal{D}_2^R, \
R(\mathcal{D}_1^R)=\mathcal{D}_1^R, \ R(\mathcal{D}_2^R)\subset
(\ker F_*)^{\perp},
$$
where $\mathcal{D}_2^R$ is the orthogonal complement of
$\mathcal{D}_1^R$ in $\ker F_*$ \cite{P2}.

We call such a basis $\{ I,J,K \}$ an {\em almost h-semi-invariant
basis}.

Throughout this paper, we will use the above notations.

\section{Almost h-conformal semi-invariant submersions}\label{semi}

In this section, we define h-conformal semi-invariant submersions
and almost h-conformal semi-invariant submersions from almost
 quaternionic Hermitian manifolds onto Riemannian manifolds.
 And we study their properties: the integrability of distributions,
 the geometry of foliations, the conditions for such maps to be
 totally geodesic, etc.

\begin{definition}
Let $(M,E,g_M)$ be an almost quaternionic Hermitian manifold and
$(N,g_N)$ a Riemannian manifold. A horizontally conformal submersion
$F : (M,E,g_M) \mapsto (N,g_N)$ is called an {\em h-conformal
semi-invariant submersion} if given a point $p\in M$ with a
neighborhood $U$, there exists a quaternionic Hermitian basis $\{
I,J,K \}$ of sections of $E$ on $U$ such that for any $R\in \{ I,J,K
\}$, there is a distribution $\mathcal{D}_1 \subset \ker F_*$ on $U$
such that
$$
\ker F_* =\mathcal{D}_1\oplus \mathcal{D}_2, \
R(\mathcal{D}_1)=\mathcal{D}_1, \ R(\mathcal{D}_2)\subset (\ker
F_*)^{\perp},
$$
where $\mathcal{D}_2$ is the orthogonal complement of
$\mathcal{D}_1$ in $\ker F_*$.
\end{definition}

We call such a basis $\{ I,J,K \}$ an {\em h-conformal
semi-invariant basis}.

\begin{definition}
Let $(M,E,g_M)$ be an almost quaternionic Hermitian manifold and
$(N,g_N)$ a Riemannian manifold. A horizontally conformal submersion
$F : (M,E,g_M) \mapsto (N,g_N)$ is called an {\em almost h-conformal
semi-invariant submersion} if given a point $p\in M$ with a
neighborhood $U$, there exists a quaternionic Hermitian basis $\{
I,J,K \}$ of sections of $E$ on $U$ such that for each $R\in \{
I,J,K \}$, there is a distribution $\mathcal{D}_1^R \subset \ker
F_*$ on $U$ such that
$$
\ker F_* =\mathcal{D}_1^R\oplus \mathcal{D}_2^R, \
R(\mathcal{D}_1^R)=\mathcal{D}_1^R, \ R(\mathcal{D}_2^R)\subset
(\ker F_*)^{\perp},
$$
where $\mathcal{D}_2^R$ is the orthogonal complement of
$\mathcal{D}_1^R$ in $\ker F_*$.
\end{definition}

We call such a basis $\{ I,J,K \}$ an {\em almost h-conformal
semi-invariant basis}.

\begin{remark}
(1) Let $F$ be an h-conformal semi-invariant submersion from a
hyperk\"{a}hler manifold $(M,I,J,K,g_M)$ onto a Riemannian manifold
$(N,g_N)$ such that $(I,J,K)$ is an h-conformal semi-invariant
basis. Then the fibers of the map $F$ are quaternionic
CR-submanifolds \cite{BCU}.

(2) Let $F : (M,E,g_M) \mapsto (N,g_N)$ be an h-conformal
semi-invariant submersion. Then the map $F$ is also an almost
h-conformal semi-invariant submersion.
\end{remark}

\noindent Let $F : (M, E, g_M) \mapsto (N, g_N)$ be an almost
h-conformal semi-invariant submersion with an almost h-conformal
semi-invariant basis $\{ I,J,K \}$.

Denote the orthogonal complement of $R\mathcal{D}_2^R$ in $(\ker
F_*)^{\perp}$ by $\mu^R$ for $R\in \{ I,J,K \}$. We easily see that
$\mu^R$ is $R$-invariant for $R\in \{ I,J,K \}$.

Then given $X\in \Gamma(\ker F_*)$, we write
\begin{equation}\label{eq8}
RX=\phi_R X+\omega_R X,
\end{equation}
where $\phi_R X\in \Gamma(\mathcal{D}_1^R)$ and $\omega_R X\in
\Gamma(R\mathcal{D}_2^R)$ for $R\in \{ I,J,K \}$.

Given $Z\in \Gamma((\ker F_*)^{\perp})$, we get
\begin{equation}\label{eq9}
RZ=B_R Z+C_R Z,
\end{equation}
where $B_R Z\in \Gamma(\mathcal{D}_2^R)$ and $C_R Z\in
\Gamma(\mathcal{\mu}^R)$ for $R\in \{ I,J,K \}$.

We see that
\begin{equation}\label{eqn91}
(\ker F_*)^{\perp} = R\mathcal{D}_2^R \oplus \mu^R \quad \text{for}
\ R\in \{ I,J,K \}
\end{equation}
and
\begin{equation}\label{eqn92}
g_M (C_R X, RV) = 0
\end{equation}
for $X\in \Gamma((\ker F_*)^{\perp})$ and $V\in
\Gamma(\mathcal{D}_2^R)$.

Define
\begin{equation}\label{eq10}
(\nabla_X \phi_R)Y := \widehat{\nabla}_X \phi_R Y-\phi_R
\widehat{\nabla}_X Y
\end{equation}
and
\begin{equation}\label{eq11}
(\nabla_X \omega_R)Y := \mathcal{H}\nabla_X \omega_R
Y-\omega_R\widehat{\nabla}_X Y
\end{equation}
for $X,Y\in \Gamma(\ker F_*)$ and $R\in \{ I,J,K \}$.

Then we easily obtain

\begin{lemma}\label{lemm1}
Let $F$ be an almost h-conformal semi-invariant submersion from a
hyperk\"{a}hler manifold $(M,I,J,K,g_M)$ onto a Riemannian manifold
$(N, g_N)$ such that $(I,J,K)$ is an almost h-conformal
semi-invariant basis. Then we get

\begin{enumerate}
\item
\begin{align*}
  &\widehat{\nabla}_X \phi_R Y+\mathcal{T}_X \omega_R Y = \phi_R\widehat{\nabla}_X Y+B_R\mathcal{T}_X Y    \\
  &\mathcal{T}_X \phi_R Y+\mathcal{H}\nabla_X \omega_R Y =
  \omega_R\widehat{\nabla}_X Y+C_R\mathcal{T}_X Y
\end{align*}
for $X,Y\in \Gamma(\ker F_*)$ and $R\in \{ I,J,K \}$.
\item
\begin{align*}
  &\mathcal{V}\nabla_Z B_RW+\mathcal{A}_Z C_RW = \phi_R\mathcal{A}_Z W+B_R\mathcal{H}\nabla_Z W    \\
  &\mathcal{A}_Z B_RW+\mathcal{H}\nabla_Z C_RW = \omega_R\mathcal{A}_Z
  W+C_R\mathcal{H}\nabla_Z W
\end{align*}
for $Z,W\in \Gamma((\ker F_*)^{\perp})$ and $R\in \{ I,J,K \}$.
\item
\begin{align*}
  &\widehat{\nabla}_X B_RZ+\mathcal{T}_X C_RZ = \phi_R\mathcal{T}_X Z+B_R\mathcal{H}\nabla_X Z    \\
  &\mathcal{T}_X B_RZ+\mathcal{H}\nabla_X C_RZ =
  \omega_R\mathcal{T}_X Z+C_R\mathcal{H}\nabla_X Z
\end{align*}
for $X\in \Gamma(\ker F_*)$, $Z\in \Gamma((\ker F_*)^{\perp})$, and
$R\in \{ I,J,K \}$.
\end{enumerate}
\end{lemma}

\begin{remark}
By (\ref{eq10}), (\ref{eq11}), and Lemma \ref{lemm1} (1), we have
\begin{equation}\label{eqn10}
(\nabla_X \omega_R)Y = B_R\mathcal{T}_X Y - \mathcal{T}_X \omega_R Y
\end{equation}
\begin{equation}\label{eqn11}
(\nabla_X \omega_R)Y = C_R\mathcal{T}_X Y - \mathcal{T}_X \phi_R Y
\end{equation}
for $X,Y\in \Gamma(\ker F_*)$ and $R\in \{ I,J,K \}$.
\end{remark}

Now, we investigate the integrability of some distributions.

\begin{lemma}
Let $F$ be an h-conformal semi-invariant submersion from a
hyperk\"{a}hler manifold $(M,I,J,K,g_M)$ onto a Riemannian manifold
$(N, g_N)$ such that $(I,J,K)$ is an h-conformal semi-invariant
basis. Then we have

(i) the distribution $\mathcal{D}_2$ is always integrable.

(ii) the following conditions are equivalent:

(a) the distribution $\mathcal{D}_1$ is  integrable.

(b) $(\nabla F_*)(W,IV) - (\nabla F_*)(V,IW)\in \Gamma(F_* \mu^I)$
for $V,W\in \Gamma(\mathcal{D}_1)$.

(c) $(\nabla F_*)(W,JV) - (\nabla F_*)(V,JW)\in \Gamma(F_* \mu^J)$
for $V,W\in \Gamma(\mathcal{D}_1)$.

(b) $(\nabla F_*)(W,KV) - (\nabla F_*)(V,KW)\in \Gamma(F_* \mu^K)$
for $V,W\in \Gamma(\mathcal{D}_1)$
\end{lemma}

\begin{proof}
By (\ref{eq5}), we have $[V,W]\in \Gamma(\ker F_*)$ for $V,W\in
\Gamma(\ker F_*)$.

We claim that $\mathcal{T}_V RW = \mathcal{T}_W RV$ for $V,W\in
\Gamma(\mathcal{D}_2)$ and $R\in \{ I,J,K \}$.

Given $X\in \Gamma(\ker F_*)$, we get
\begin{align*}
  g_M (\mathcal{T}_V RW, X)
  &= -g_M (RW, \nabla_V X) = -g_M (RW, \nabla_X V)
  = g_M (\nabla_X RW, V)  \\
  &= -g_M (\nabla_X W, RV) = -g_M (\nabla_W X, RV) = g_M (X, \nabla_W RV)   \\
  &= g_M (X, \mathcal{T}_W RV),
\end{align*}
which means our claim.

Given $V,W\in \Gamma(\mathcal{D}_2)$ and $Z\in
\Gamma(\mathcal{D}_1)$, we obtain
$$
g_M ([V,W], Z) = g_M (\nabla_V W - \nabla_W V, Z) = g_M
(\mathcal{T}_V RW - \mathcal{T}_W RV, RZ) = 0,
$$
which implies (i).

For (ii), given $V,W\in \Gamma(\mathcal{D}_1)$, $Z\in
\Gamma(\mathcal{D}_2)$, and $R\in \{ I,J,K \}$, we have
\begin{align*}
  g_M ([V,W], Z)
  &= \frac{1}{\lambda^2} g_N (F_* \nabla_V RW - F_* \nabla_W RV, F_* RZ)  \\
  &= \frac{1}{\lambda^2} g_N ((\nabla F_*)(W,RV) - (\nabla F_*)(V,RW), F_* RZ)
\end{align*}
so that we get $(a) \Leftrightarrow (b)$, $(a) \Leftrightarrow (c)$,
 $(a) \Leftrightarrow (d)$.

Therefore, the result follows.
\end{proof}

\begin{theorem}\label{int1}
Let $F$ be an almost h-conformal semi-invariant submersion from a
hyperk\"{a}hler manifold $(M,I,J,K,g_M)$ onto a Riemannian manifold
$(N, g_N)$ such that $(I,J,K)$ is an almost h-conformal
semi-invariant basis. Then the following conditions are equivalent:

(a) the distribution $(\ker F_*)^{\perp}$ is integrable.

(b) $\mathcal{A}_Y \omega_I B_I X - \mathcal{A}_X \omega_I B_I Y +
\phi_I (\mathcal{A}_Y C_I X - \mathcal{A}_X C_I Y)\in
\Gamma(\mathcal{D}_2^I)$ and
\begin{align*}
  &\frac{1}{\lambda^2} g_N (\nabla_Y^F F_* C_I X - \nabla_X^F F_* C_I Y, F_* IV)  \\
  &= g_M (\mathcal{A}_Y B_I X - \mathcal{A}_X B_I Y - C_I Y(\ln \lambda)X + C_I X(\ln \lambda)Y   \\
  &+ 2g_M (X, C_I Y)\nabla(\ln \lambda), IV)
\end{align*}
for $X,Y\in \Gamma((\ker F_*)^{\perp})$ and $V\in
\Gamma(\mathcal{D}_2^I)$.

(c) $\mathcal{A}_Y \omega_J B_J X - \mathcal{A}_X \omega_J B_J Y +
\phi_J (\mathcal{A}_Y C_J X - \mathcal{A}_X C_J Y)\in
\Gamma(\mathcal{D}_2^J)$ and
\begin{align*}
  &\frac{1}{\lambda^2} g_N (\nabla_Y^F F_* C_J X - \nabla_X^F F_* C_J Y, F_* JV)  \\
  &= g_M (\mathcal{A}_Y B_J X - \mathcal{A}_X B_J Y - C_J Y(\ln \lambda)X + C_J X(\ln \lambda)Y   \\
  &+ 2g_M (X, C_J Y)\nabla(\ln \lambda), JV)
\end{align*}
for $X,Y\in \Gamma((\ker F_*)^{\perp})$ and $V\in
\Gamma(\mathcal{D}_2^J)$.

(d) $\mathcal{A}_Y \omega_K B_K X - \mathcal{A}_X \omega_K B_K Y +
\phi_K (\mathcal{A}_Y C_K X - \mathcal{A}_X C_K Y)\in
\Gamma(\mathcal{D}_2^K)$ and
\begin{align*}
  &\frac{1}{\lambda^2} g_N (\nabla_Y^F F_* C_K X - \nabla_X^F F_* C_K Y, F_* KV)  \\
  &= g_M (\mathcal{A}_Y B_K X - \mathcal{A}_X B_K Y - C_K Y(\ln \lambda)X + C_K X(\ln \lambda)Y   \\
  &+ 2g_M (X, C_K Y)\nabla(\ln \lambda), KV)
\end{align*}
for $X,Y\in \Gamma((\ker F_*)^{\perp})$ and $V\in
\Gamma(\mathcal{D}_2^K)$.
\end{theorem}

\begin{proof}
Given $X,Y\in \Gamma((\ker F_*)^{\perp})$, $W\in
\Gamma(\mathcal{D}_1^R)$, and $R\in \{ I,J,K \}$, we have
\begin{align*}
  g_M ([X,Y], W)
  &= g_M (\nabla_X B_R Y, RW) + g_M (\nabla_X C_R Y, RW)  \\
  &- g_M (\nabla_Y B_R X, RW) - g_M (\nabla_Y C_R X, RW)  \\
  &= -g_M (\nabla_X RB_R Y, W) + g_M (\mathcal{A}_X C_R Y, RW)  \\
  &+ g_M (\nabla_Y RB_R X, W) - g_M (\mathcal{A}_Y C_R X, RW)  \\
  &= -g_M (\nabla_X \omega_R B_R Y, W) - g_M (\phi_R\mathcal{A}_X C_R Y, W)  \\
  &+ g_M (\nabla_Y \omega_R B_R X, W) + g_M (\phi_R\mathcal{A}_Y C_R X, W) \ (\text{since} \ \phi_R B_R = 0)  \\
  &= g_M (\mathcal{A}_Y \omega_R B_R X - \mathcal{A}_X \omega_R B_R Y + \phi_R\mathcal{A}_Y C_R X - \phi_R\mathcal{A}_X C_R Y, W)
\end{align*}
so that
\begin{eqnarray}
  &&g_M ([X,Y], W) = 0 \quad \text{for} \ W\in \Gamma(\mathcal{D}_1^R)   \label{eqn177}  \\
  &&\Leftrightarrow \mathcal{A}_Y \omega_R B_R X - \mathcal{A}_X \omega_R B_R Y + \phi_R\mathcal{A}_Y C_R X - \phi_R\mathcal{A}_X C_R
  Y \in \Gamma(\mathcal{D}_2^R). \nonumber
\end{eqnarray}
Given $V\in \Gamma(\mathcal{D}_2^R)$, by using (\ref{eqn77}) and
(\ref{eqn92}), we get
\begin{align*}
  g_M ([X,Y], V)
  &= g_M (\nabla_X B_R Y, RV) + g_M (\nabla_X C_R Y, RV)  \\
  &- g_M (\nabla_Y B_R X, RV) - g_M (\nabla_Y C_R X, RV)  \\
  &= g_M (\mathcal{A}_X B_R Y - \mathcal{A}_Y B_R X, RV)  \\
  &+ \frac{1}{\lambda^2}g_N (-X(\ln \lambda)F_* C_R Y - C_R Y(\ln \lambda)F_* X   \\
  &+ g_M (X, C_R Y)F_* \nabla(\ln \lambda) + \nabla_X^F F_* C_R Y   \\
  &+ Y(\ln \lambda)F_* C_R X + C_R X(\ln \lambda)F_* Y - g_M (Y, C_R X)F_* \nabla(\ln \lambda)   \\
  &- \nabla_Y^F F_* C_R X, F_* RV)  \\
  &= g_M (\mathcal{A}_X B_R Y - \mathcal{A}_Y B_R X + C_R X(\ln \lambda)Y - C_R Y(\ln \lambda)X   \\
  &+ 2g_M (X, C_R Y)\nabla(\ln \lambda), RV) \\
  &-\frac{1}{\lambda^2}g_N (\nabla_Y^F F_* C_R X - \nabla_X^F F_* C_R Y, F_* RV)
\end{align*}
so that
\begin{eqnarray}
  &&g_M ([X,Y], V) = 0 \quad \text{for} \ V\in \Gamma(\mathcal{D}_2^R)  \label{eqn178}  \\
  &&\Leftrightarrow \frac{1}{\lambda^2}g_N (\nabla_Y^F F_* C_R X - \nabla_X^F F_* C_R Y, F_* RV) \nonumber   \\
  &&= g_M (\mathcal{A}_X B_R Y - \mathcal{A}_Y B_R X + C_R X(\ln \lambda)Y - C_R Y(\ln \lambda)X   \nonumber   \\
  &&+ 2g_M (X, C_R Y)\nabla(\ln \lambda), RV).  \nonumber
\end{eqnarray}
Using (\ref{eqn177}) and (\ref{eqn178}), we obtain $(a)
\Leftrightarrow (b)$, $(a) \Leftrightarrow (c)$, $(a)
\Leftrightarrow (d)$.

Therefore, we have the result.
\end{proof}

\begin{theorem}
Let $F$ be an almost h-conformal semi-invariant submersion from a
hyperk\"{a}hler manifold $(M,I,J,K,g_M)$ onto a Riemannian manifold
$(N, g_N)$ such that $(I,J,K)$ is an almost h-conformal
semi-invariant basis. Assume that the distribution $(\ker
F_*)^{\perp}$ is integrable. Then the following conditions are
equivalent:

(a) the map $F$ is horizontally homothetic.

(b) $\lambda^2 g_M (\mathcal{A}_Y B_I X - \mathcal{A}_X B_I Y, IV) =
g_N (\nabla_Y^F F_* C_I X - \nabla_X^F F_* C_I Y, F_* IV)$ for
$X,Y\in \Gamma((\ker F_*)^{\perp})$ and $V\in
\Gamma(\mathcal{D}_2^I)$.

(c) $\lambda^2 g_M (\mathcal{A}_Y B_J X - \mathcal{A}_X B_J Y, JV) =
g_N (\nabla_Y^F F_* C_J X - \nabla_X^F F_* C_J Y, F_* JV)$ for
$X,Y\in \Gamma((\ker F_*)^{\perp})$ and $V\in
\Gamma(\mathcal{D}_2^J)$.

(d) $\lambda^2 g_M (\mathcal{A}_Y B_K X - \mathcal{A}_X B_K Y, KV) =
g_N (\nabla_Y^F F_* C_K X - \nabla_X^F F_* C_K Y, F_* KV)$ for
$X,Y\in \Gamma((\ker F_*)^{\perp})$ and $V\in
\Gamma(\mathcal{D}_2^K)$.
\end{theorem}

\begin{proof}
Given $X,Y\in \Gamma((\ker F_*)^{\perp})$, $V\in
\Gamma(\mathcal{D}_2^R)$, and $R\in \{ I,J,K \}$, from the proof of
Theorem \ref{int1}, we have
\begin{eqnarray}
  g_M ([X,Y], V)
  &=&g_M (\mathcal{A}_X B_R Y - \mathcal{A}_Y B_R X + C_R X(\ln \lambda)Y   \label{eq12} \\
  &&-C_R Y(\ln \lambda)X + 2g_M (X, C_R Y)\nabla(\ln \lambda), RV)  \nonumber \\
  &&-\frac{1}{\lambda^2}g_N (\nabla_Y^F F_* C_R X - \nabla_X^F F_* C_R Y, F_*
  RV).  \nonumber
\end{eqnarray}
Using (\ref{eq12}), it is easy to see $(a) \Rightarrow (b)$, $(a)
\Rightarrow (c)$, $(a) \Rightarrow (d)$.

Conversely, from (\ref{eq12}), we get
\begin{equation}\label{eq13}
g_M (C_R X(\ln \lambda)Y - C_R Y(\ln \lambda)X + 2g_M (X, C_R
Y)\nabla(\ln \lambda), RV) = 0
\end{equation}
Applying $Y = RV$ at (\ref{eq13}), we obtain
$$
g_M (\nabla(\ln \lambda), C_R X) g_M (RV, RV) = 0,
$$
which implies
\begin{equation}\label{eq14}
g_M (\nabla(\lambda), X) = 0 \quad \text{for} \ X\in \Gamma(\mu^R).
\end{equation}
Applying $Y = C_R X$, $X\in \Gamma(\mu^R)$, at (\ref{eq13}), we have
$$
2g_M (X, C_R^2 X) g_M (\nabla(\ln \lambda), RV) = -2g_M (X, X) g_M
(\nabla(\ln \lambda), RV) = 0,
$$
which implies
\begin{equation}\label{eq15}
g_M (\nabla(\lambda), RV) = 0  \quad \text{for} \  V\in
\Gamma(\mathcal{D}_2^R).
\end{equation}
By (\ref{eq14}) and (\ref{eq15}), we get $(b) \Rightarrow (a)$, $(c)
\Rightarrow (a)$, $(d) \Rightarrow (a)$.

Therefore, the result follows.
\end{proof}

We deal with some particular type of conformal submersions.

\begin{definition}
Let $F$ be an almost h-conformal semi-invariant submersion from an
almost quaternionic Hermitian manifold $(M,E,g_M)$ onto a Riemannian
manifold $(N,g_N)$. If $R(\mathcal{D}_2^R) = (\ker F_*)^{\perp}$ for
$R\in \{ I,K \}$ and $J(\ker F_*) = \ker F_*$ (i.e.,
$\mathcal{D}_2^J = \{ 0 \}$), then we call the map $F$ an {\em
almost h-conformal anti-holomorphic semi-invariant submersion}
\end{definition}

We call such a basis $\{ I,J,K \}$ an {\em almost h-conformal
anti-holomorphic semi-invariant basis}.

\begin{remark}
(1) We easily see that $J(\ker F_*) = \ker F_*$ implies $J((\ker
F_*)^{\perp}) = (\ker F_*)^{\perp}$.

(2) Let $F : (M,E,g_M) \mapsto (N,g_N)$ be an h-conformal
semi-invariant submersion. Then it is not possible to get
$R(\mathcal{D}_2) = (\ker F_*)^{\perp}$ for $R\in \{ I,J,K \}$. If
not, then $K(\mathcal{D}_2) = (\ker F_*)^{\perp}$ and
$K(\mathcal{D}_2) = IJ(\mathcal{D}_2) = I((\ker F_*)^{\perp}) =
\mathcal{D}_2$, contradiction!

So, our definition makes sense and its example is Example
\ref{exam1}.
\end{remark}

\begin{corollary}
Let $F$ be an almost h-conformal anti-holomorphic semi-invariant
submersion from a hyperk\"{a}hler manifold $(M,I,J,K,g_M)$ onto a
Riemannian manifold $(N, g_N)$ such that $(I,J,K)$ is an almost
h-conformal anti-holomorphic semi-invariant basis. Then the
following conditions are equivalent:

(a) the distribution $(\ker F_*)^{\perp}$ is integrable.

(b) $\mathcal{A}_{IV_1} IV_2 = \mathcal{A}_{IV_2} IV_1$ for
$V_1,V_2\in \Gamma(\mathcal{D}_2^I)$.

(c) $\mathcal{A}_{KV_1} KV_2 = \mathcal{A}_{KV_2} KV_1$ for
$V_1,V_2\in \Gamma(\mathcal{D}_2^K)$.
\end{corollary}

\begin{proof}
We see that $C_R = 0$, $B_R = R$ on $(\ker F_*)^{\perp}$ and
$\omega_R = R$ on $\mathcal{D}_2^R$ for $R\in \{ I,K \}$.

Applying $X = RV_1$ and $Y = RV_2$, $V_1,V_2\in
\Gamma(\mathcal{D}_2^R)$, at Theorem \ref{int1}, we have
$$
\mathcal{A}_{RV_1} RV_2 - \mathcal{A}_{RV_2} RV_1 \in
\Gamma(\mathcal{D}_2^R)
$$
and
$$
0 = g_M (\mathcal{A}_{RV_2} RV_1 - \mathcal{A}_{RV_1} RV_2, V) \quad
\text{for} \ V\in \Gamma(\mathcal{D}_2^R),
$$
which are equivalent to
$$
\mathcal{A}_{RV_1} RV_2 = \mathcal{A}_{RV_2} RV_1 \quad \text{for} \
V_1,V_2\in \Gamma(\mathcal{D}_2^R).
$$
Hence, we get $(a) \Leftrightarrow (b)$, $(a) \Leftrightarrow (c)$.

Therefore, we obtain the result.
\end{proof}

We consider the geometry of foliations and the condition for such
maps to be horizontally homothetic.

\begin{theorem}\label{thm2}
Let $F$ be an almost h-conformal semi-invariant submersion from a
hyperk\"{a}hler manifold $(M,I,J,K,g_M)$ onto a Riemannian manifold
$(N, g_N)$ such that $(I,J,K)$ is an almost h-conformal
semi-invariant basis. Then the following conditions are equivalent:

(a) the distribution $(\ker F_*)^{\perp}$ defines a totally geodesic
foliation on $M$.

(b) $\mathcal{A}_X C_I Y + \mathcal{V}\nabla_X B_I Y\in
\Gamma(\mathcal{D}_2^I)$ and
$$
g_N (\nabla_X^F F_* IV, F_* C_I V) =
\lambda^2 g_M (\mathcal{A}_X B_I Y - C_I Y(\ln \lambda)X + g_M (X,
C_I Y) \nabla(\ln \lambda), IV)
$$
for $X,Y\in \Gamma((\ker F_*)^{\perp})$ and $V\in
\Gamma(\mathcal{D}_2^I)$.

(c) $\mathcal{A}_X C_J Y + \mathcal{V}\nabla_X B_J Y\in
\Gamma(\mathcal{D}_2^J)$ and
$$
g_N (\nabla_X^F F_* JV, F_* C_J V) = \lambda^2 g_M (\mathcal{A}_X
B_J Y - C_J Y(\ln \lambda)X + g_M (X, C_J Y) \nabla(\ln \lambda),
JV)
$$
for $X,Y\in \Gamma((\ker F_*)^{\perp})$ and $V\in
\Gamma(\mathcal{D}_2^J)$.

(d) $\mathcal{A}_X C_K Y + \mathcal{V}\nabla_X B_K Y\in
\Gamma(\mathcal{D}_2^K)$ and
$$
g_N (\nabla_X^F F_* KV, F_* C_K V) = \lambda^2 g_M (\mathcal{A}_X
B_K Y - C_K Y(\ln \lambda)X + g_M (X, C_K Y) \nabla(\ln \lambda),
KV)
$$
for $X,Y\in \Gamma((\ker F_*)^{\perp})$ and $V\in
\Gamma(\mathcal{D}_2^K)$.
\end{theorem}

\begin{proof}
Given $X,Y\in \Gamma((\ker F_*)^{\perp})$, $W\in
\Gamma(\mathcal{D}_1^R)$, and $R\in \{ I,J,K \}$, we obtain
$$
g_M (\nabla_X Y, W) = -g_M (\phi(\mathcal{A}_X C_R Y +
\mathcal{V}\nabla_X B_R Y), W)
$$
so that
\begin{equation}\label{eq16}
g_M (\nabla_X Y, W) = 0 \Leftrightarrow \mathcal{A}_X C_R Y +
\mathcal{V}\nabla_X B_R Y \in \Gamma(\mathcal{D}_2^R).
\end{equation}
Given $V\in \Gamma(\mathcal{D}_2^R)$, by using (\ref{eqn77}) and
(\ref{eqn92}),  we have
\begin{align*}
  g_M (\nabla_X Y, V)
  &= g_M (\mathcal{A}_X B_R Y, RV) - g_M (C_R Y, \nabla_X RV)  \\
  &= g_M (\mathcal{A}_X B_R Y, RV) + \frac{1}{\lambda^2}g_N (F_* C_R Y, RV(\ln \lambda)F_* X   \\
  &- g_M (X, RV)F_* \nabla(\ln \lambda) - \nabla_X^F F_* RV)  \\
  &= g_M (\mathcal{A}_X B_R Y + g_M (C_R Y, X)\nabla(\ln \lambda) - C_R Y(\ln \lambda)X, RV)    \\
  &- \frac{1}{\lambda^2}g_N (F_* C_R Y, \nabla_X^F F_* RV)
\end{align*}
so that
\begin{eqnarray}
  &&g_M (\nabla_X Y, V) = 0   \label{eq17}  \\
  &&\Leftrightarrow g_N (F_* C_R Y, \nabla_X^F F_* RV) =  \lambda^2g_M (\mathcal{A}_X B_R Y  \nonumber    \\
  &&+ g_M (C_R Y, X)\nabla(\ln \lambda) - C_R Y(\ln \lambda)X, RV). \nonumber
\end{eqnarray}
By (\ref{eq16}) and (\ref{eq17}), we get $(a) \Leftrightarrow (b)$,
$(a) \Leftrightarrow (c)$, $(a) \Leftrightarrow (d)$.

Therefore, the result follows.
\end{proof}

We introduce another notion on distributions and investigate it.

\begin{definition}
Let $F$ be an almost h-conformal semi-invariant submersion from a
hyperk\"{a}hler manifold $(M,I,J,K,g_M)$ onto a Riemannian manifold
$(N, g_N)$ such that $(I,J,K)$ is an almost h-conformal
semi-invariant basis. Given $R\in \{ I,J,K \}$, we call the
distribution $\mathcal{D}_2^R$ {\em parallel along $(\ker
F_*)^{\perp}$} if $\nabla_X V\in \Gamma(\mathcal{D}_2^R)$ for $X\in
\Gamma((\ker F_*)^{\perp})$ and $V\in \Gamma(\mathcal{D}_2^R)$.
\end{definition}

\begin{lemma}
Let $F$ be an almost h-conformal semi-invariant submersion from a
hyperk\"{a}hler manifold $(M,I,J,K,g_M)$ onto a Riemannian manifold
$(N, g_N)$ such that $(I,J,K)$ is an almost h-conformal
semi-invariant basis. Assume that the distribution $\mathcal{D}_2^R$
is parallel along $(\ker F_*)^{\perp}$ for $R\in \{ I,J,K \}$. Then
the following conditions are equivalent:

(a) the map $F$ is horizontally homothetic.

(b)
$$
\lambda^2g_M (\mathcal{A}_X B_I Y, IV) = g_N (\nabla_X^F F_* IV, F_*
C_I Y)
$$
for $X,Y\in \Gamma((\ker F_*)^{\perp})$ and $V\in
\Gamma(\mathcal{D}_2^I)$.

(c)
$$
\lambda^2g_M (\mathcal{A}_X B_J Y, JV) = g_N (\nabla_X^F F_* JV, F_*
C_J Y)
$$
for $X,Y\in \Gamma((\ker F_*)^{\perp})$ and $V\in
\Gamma(\mathcal{D}_2^J)$.

(d)
$$
\lambda^2g_M (\mathcal{A}_X B_K Y, KV) = g_N (\nabla_X^F F_* KV, F_*
C_K Y)
$$
for $X,Y\in \Gamma((\ker F_*)^{\perp})$ and $V\in
\Gamma(\mathcal{D}_2^K)$.
\end{lemma}

\begin{proof}
Given $X,Y\in \Gamma((\ker F_*)^{\perp})$, $V\in
\Gamma(\mathcal{D}_2^R)$, and $R\in \{ I,J,K \}$, by the proof of
Theorem \ref{thm2},  we have
\begin{eqnarray}
  g_M (\nabla_X Y, V)
  &=& g_M (\mathcal{A}_X B_R Y + g_M (C_R Y, X)\nabla(\ln \lambda)  \label{eq18}   \\
  &&-C_R Y(\ln \lambda)X, RV) - \frac{1}{\lambda^2}g_N (F_* C_R Y, \nabla_X^F F_* RV). \nonumber
\end{eqnarray}
Since $g_M (\nabla_X Y, V) = -g_M (Y, \nabla_X V) = 0$, from
(\ref{eq18}), we get $(a) \Rightarrow (b)$, $(a) \Rightarrow (c)$,
$(a) \Rightarrow (d)$.

Conversely, from (\ref{eq18}), we obtain
\begin{equation}\label{eq19}
-g_M (C_R Y, \nabla(\ln \lambda)) g_M (X, RV) + g_M (X, C_R Y) g_M
(\nabla(\ln \lambda), RV) = 0.
\end{equation}
Applying $X = RV$ at (\ref{eq19}), we have
$$
-g_M (C_R Y, \nabla(\ln \lambda)) g_M (RV, RV) = 0,
$$
which implies
\begin{equation}\label{eq20}
g_M (X, \nabla(\lambda)) = 0 \quad \text{for} \ X\in \Gamma(\mu^R).
\end{equation}
Applying $X = C_R Y$ at (\ref{eq19}), we get
$$
g_M (C_R Y, C_R Y) g_M (\nabla(\ln \lambda), RV) = 0,
$$
which implies
\begin{equation}\label{eq21}
g_M (\nabla(\lambda), RV) = 0 \quad \text{for} \ V\in
\Gamma(\mathcal{D}_2^R).
\end{equation}
Using (\ref{eq20}) and (\ref{eq21}), we obtain $(b) \Rightarrow
(a)$, $(c) \Rightarrow (a)$, $(d) \Rightarrow (a)$.

Therefore, the result follows.
\end{proof}

\begin{lemma}
Let $F$ be an almost h-conformal anti-holomorphic semi-invariant
submersion from a hyperk\"{a}hler manifold $(M,I,J,K,g_M)$ onto a
Riemannian manifold $(N, g_N)$ such that $(I,J,K)$ is an almost
h-conformal anti-holomorphic semi-invariant basis. Then the
following conditions are equivalent:

(a) the distribution $(\ker F_*)^{\perp}$ defines a totally geodesic
foliation on $M$.

(b) the distribution $\mathcal{D}_2^I$ is parallel along $(\ker
F_*)^{\perp}$.

(c) the distribution $\mathcal{D}_2^K$ is parallel along $(\ker
F_*)^{\perp}$.
\end{lemma}

\begin{proof}
We see that $B_R = R$ and $C_R = 0$ on $(\ker F_*)^{\perp}$ for
$R\in \{ I,K \}$.

Given $X,Y\in \Gamma((\ker F_*)^{\perp})$ and $V\in
\Gamma(\mathcal{D}_2^R)$, from Theorem \ref{thm2},  we have
\begin{align*}
  (a)
  &\Leftrightarrow \mathcal{V}\nabla_X RY\in \Gamma(\mathcal{D}_2^R) \ \text{and} \ g_M (\mathcal{A}_X RY, RV) = 0  \\
  &\Leftrightarrow \nabla_X RY\in \Gamma(\mathcal{D}_2^R).
\end{align*}
Hence, we get $(a) \Leftrightarrow (b)$, $(a) \Leftrightarrow (c)$.

Therefore, we obtain the result.
\end{proof}

\begin{theorem}\label{thm3}
Let $F$ be an almost h-conformal semi-invariant submersion from a
hyperk\"{a}hler manifold $(M,I,J,K,g_M)$ onto a Riemannian manifold
$(N, g_N)$ such that $(I,J,K)$ is an almost h-conformal
semi-invariant basis. Then the following conditions are equivalent:

(a) the distribution $\ker F_*$ defines a totally geodesic foliation
on $M$.

(b) $\mathcal{T}_V \omega_I U + \widehat{\nabla}_V \phi_I U\in
\Gamma(\mathcal{D}_1^I)$ and
$$
g_N (\nabla_{\omega_I V}^F F_* X, F_* \omega_I U) = \lambda^2g_M
(C_I \mathcal{T}_U \phi_I V + \mathcal{A}_{\omega_I V} \phi_I U +
g_M (\omega_I V, \omega_I U)\nabla(\ln \lambda), X)
$$
for $U,V\in \Gamma(\ker F_*)$ and $X\in \Gamma(\mu^I)$.

(c) $\mathcal{T}_V \omega_J U + \widehat{\nabla}_V \phi_J U\in
\Gamma(\mathcal{D}_1^J)$ and
$$
g_N (\nabla_{\omega_J V}^F F_* X, F_* \omega_J U) = \lambda^2g_M
(C_J \mathcal{T}_U \phi_J V + \mathcal{A}_{\omega_J V} \phi_J U +
g_M (\omega_J V, \omega_J U)\nabla(\ln \lambda), X)
$$
for $U,V\in \Gamma(\ker F_*)$ and $X\in \Gamma(\mu^J)$.

(d) $\mathcal{T}_V \omega_K U + \widehat{\nabla}_V \phi_K U\in
\Gamma(\mathcal{D}_1^K)$ and
$$
g_N (\nabla_{\omega_K V}^F F_* X, F_* \omega_K U) = \lambda^2g_M
(C_K \mathcal{T}_U \phi_K V + \mathcal{A}_{\omega_K V} \phi_K U +
g_M (\omega_K V, \omega_K U)\nabla(\ln \lambda), X)
$$
for $U,V\in \Gamma(\ker F_*)$ and $X\in \Gamma(\mu^K)$.
\end{theorem}

\begin{proof}
Given $U,V\in \Gamma(\ker F_*)$, $W\in \Gamma(\mathcal{D}_2^R)$, and
$R\in \{ I,J,K \}$, by using (\ref{eqn92}),  we have
$$
g_M (\nabla_V U, RW) = -g_M (\omega_R(\widehat{\nabla}_V \phi_R U +
\mathcal{T}_V \omega_R U), RW)
$$
so that
\begin{equation}\label{eq22}
g_M (\nabla_V U, RW) = 0 \Leftrightarrow \widehat{\nabla}_V \phi_R U
+ \mathcal{T}_V \omega_R U\in \Gamma(\mathcal{D}_1^R)
\end{equation}
Given $X\in \Gamma(\mu^R)$, by using (\ref{eqn5}) and (\ref{eqn91}),
we get
\begin{align*}
  &g_M (\nabla_U V, X)   \\
  &= g_M (\nabla_U \phi_R V, RX) + g_M (\phi_R U, \nabla_{\omega_R V} X) + g_M (\omega_R U, \nabla_{\omega_R V} X)  \\
  &= g_M (\mathcal{T}_U \phi_R V, RX) + g_M (\phi_R U, \mathcal{A}_{\omega_R V} X)    \\
  &- \frac{1}{\lambda^2}g_M (\nabla(\ln \lambda), X) g_N (F_* \omega_R V, F_* \omega_R U)
  + \frac{1}{\lambda^2}g_N (\nabla_{\omega_R V}^F F_* X, F_* \omega_R U)     \\
  &= g_M (-C_R \mathcal{T}_U \phi_R V - \mathcal{A}_{\omega_R V} \phi_R U - g_M (\omega_R V, \omega_R U)\nabla(\ln \lambda), X)  \\
  &+ \frac{1}{\lambda^2}g_N (\nabla_{\omega_R V}^F F_* X, F_* \omega_R U)
\end{align*}
so that
\begin{eqnarray}
  &&g_M (\nabla_U V, X) = 0  \label{eq23}  \\
  &&\Leftrightarrow g_N (\nabla_{\omega_R V}^F F_* X, F_* \omega_R U) \nonumber  \\
  &&= \lambda^2 g_M (C_R \mathcal{T}_U \phi_R V + \mathcal{A}_{\omega_R V} \phi_R U
+ g_M (\omega_R V, \omega_R U)\nabla(\ln \lambda), X). \nonumber
\end{eqnarray}
Using (\ref{eq22}) and (\ref{eq23}), we obtain $(a) \Leftrightarrow
(b)$, $(a) \Leftrightarrow (c)$, $(a) \Leftrightarrow (d)$.

Therefore, the result follows.
\end{proof}

\begin{definition}
Let $F$ be an almost h-conformal semi-invariant submersion from a
hyperk\"{a}hler manifold $(M,I,J,K,g_M)$ onto a Riemannian manifold
$(N, g_N)$ such that $(I,J,K)$ is an almost h-conformal
semi-invariant basis. Then given $R\in \{ I,J,K \}$, we call the
distribution $\mu^R$ {\em parallel along $\ker F_*$} if $\nabla_U
X\in \Gamma(\mu^R)$ for $X\in \Gamma(\mu^R)$ and $U\in \Gamma(\ker
F_*)$.
\end{definition}

\begin{lemma}
Let $F$ be an almost h-conformal semi-invariant submersion from a
hyperk\"{a}hler manifold $(M,I,J,K,g_M)$ onto a Riemannian manifold
$(N, g_N)$ such that $(I,J,K)$ is an almost h-conformal
semi-invariant basis. Assume that the distribution $\mu^R$ is
parallel along $\ker F_*$ for any $R\in \{ I,J,K \}$.

Then given $R\in \{ I,J,K \}$, the following conditions are
equivalent:

(a) dilation $\lambda$ is constant on $\mu^R$.

(b)
$$
g_N (\nabla_{\omega_R V}^F F_* X, F_* \omega_R U) = \lambda^2g_M
(C_R \mathcal{T}_U \phi_R V + \mathcal{A}_{\omega_R V} \phi_R U, X)
$$
for $X\in \Gamma(\mu^R)$ and $U,V\in \Gamma(\ker F_*)$.
\end{lemma}

\begin{proof}
Given $X\in \Gamma(\mu^R)$ and $U,V\in \Gamma(\ker F_*)$, by using
the proof of Theorem \ref{thm3} and (\ref{eqn92}), we have
\begin{align*}
  &g_M (\nabla_U V, X)   \\
  &= g_M (-C_R \mathcal{T}_U \phi_R V - \mathcal{A}_{\omega_R V} \phi_R U - g_M (\omega_R V, \omega_R U)\nabla(\ln \lambda), X)  \\
  &+ \frac{1}{\lambda^2}g_N (\nabla_{\omega_R V}^F F_* X, F_* \omega_R U)
\end{align*}
so that since $g_M (\nabla_U V, X) = -g_M (V, \nabla_U X) = 0$, it
is easy to get $(a) \Leftrightarrow (b)$.
\end{proof}

Denote by $M_{\ker F_*}$ and $M_{(\ker F_*)^{\perp}}$ the integral
manifolds of the distributions $\ker F_*$ and $(\ker F_*)^{\perp}$,
respectively.

Using Theorem \ref{thm2} and Theorem \ref{thm3}, we have

\begin{theorem}
Let $F$ be an almost h-conformal semi-invariant submersion from a
hyperk\"{a}hler manifold $(M,I,J,K,g_M)$ onto a Riemannian manifold
$(N, g_N)$ such that $(I,J,K)$ is an almost h-conformal
semi-invariant basis. Then the following conditions are equivalent:

(a) $M$ is locally a product Riemannian manifold $M_{\ker F_*}
\times M_{(\ker F_*)^{\perp}}$.

(b) $\mathcal{A}_X C_I Y + \mathcal{V}\nabla_X B_I Y\in
\Gamma(\mathcal{D}_2^I)$,

$g_N (\nabla_X^F F_* IV, F_* C_I V) = \lambda^2 g_M (\mathcal{A}_X
B_I Y - C_I Y(\ln \lambda)X + g_M (X, C_I Y) \nabla(\ln \lambda),
IV)$

for $X,Y\in \Gamma((\ker F_*)^{\perp})$, $V\in
\Gamma(\mathcal{D}_2^I)$.

$\mathcal{T}_V \omega_I U + \widehat{\nabla}_V \phi_I U\in
\Gamma(\mathcal{D}_1^I)$,

$g_N (\nabla_{\omega_I V}^F F_* X, F_* \omega_I U) = \lambda^2g_M
(C_I \mathcal{T}_U \phi_I V + \mathcal{A}_{\omega_I V} \phi_I U +
g_M (\omega_I V, \omega_I U)\nabla(\ln \lambda), X)$

for $U,V\in \Gamma(\ker F_*)$, $X\in \Gamma(\mu^I)$.

(c) $\mathcal{A}_X C_J Y + \mathcal{V}\nabla_X B_J Y\in
\Gamma(\mathcal{D}_2^J)$,

$g_N (\nabla_X^F F_* JV, F_* C_J V) = \lambda^2 g_M (\mathcal{A}_X
B_J Y - C_J Y(\ln \lambda)X + g_M (X, C_J Y) \nabla(\ln \lambda),
JV)$

for $X,Y\in \Gamma((\ker F_*)^{\perp})$ and $V\in
\Gamma(\mathcal{D}_2^J)$.

$\mathcal{T}_V \omega_J U + \widehat{\nabla}_V \phi_J U\in
\Gamma(\mathcal{D}_1^J)$,

$g_N (\nabla_{\omega_J V}^F F_* X, F_* \omega_J U) = \lambda^2g_M
(C_J \mathcal{T}_U \phi_J V + \mathcal{A}_{\omega_J V} \phi_J U +
g_M (\omega_J V, \omega_J U)\nabla(\ln \lambda), X)$

for $U,V\in \Gamma(\ker F_*)$ and $X\in \Gamma(\mu^J)$.

(d) $\mathcal{A}_X C_K Y + \mathcal{V}\nabla_X B_K Y\in
\Gamma(\mathcal{D}_2^K)$,

$g_N (\nabla_X^F F_* KV, F_* C_K V) = \lambda^2 g_M (\mathcal{A}_X
B_K Y - C_K Y(\ln \lambda)X + g_M (X, C_K Y) \nabla(\ln \lambda),
KV)$

for $X,Y\in \Gamma((\ker F_*)^{\perp})$ and $V\in
\Gamma(\mathcal{D}_2^K)$.

$\mathcal{T}_V \omega_K U + \widehat{\nabla}_V \phi_K U\in
\Gamma(\mathcal{D}_1^K)$,

$g_N (\nabla_{\omega_K V}^F F_* X, F_* \omega_K U) = \lambda^2g_M
(C_K \mathcal{T}_U \phi_K V + \mathcal{A}_{\omega_K V} \phi_K U +
g_M (\omega_K V, \omega_K U)\nabla(\ln \lambda), X)$

for $U,V\in \Gamma(\ker F_*)$ and $X\in \Gamma(\mu^K)$.
\end{theorem}

\begin{theorem}\label{thm4}
Let $F$ be an h-conformal semi-invariant submersion from a
hyperk\"{a}hler manifold $(M,I,J,K,g_M)$ onto a Riemannian manifold
$(N, g_N)$ such that $(I,J,K)$ is an h-conformal semi-invariant
basis. Then the following conditions are equivalent:

(a) the distribution $\mathcal{D}_1$ defines a totally geodesic
foliation on $M$.

(b)
\begin{align*}
  &(\nabla F_*)(V,IW)\in \Gamma(F_* \mu^I),   \\
  &\displaystyle{g_N ((\nabla F_*)(V,IW), F_* C_I X) = \lambda^2g_M (W,
\mathcal{T}_V \omega_I B_I X)}
\end{align*}
for $V,W\in \Gamma(\mathcal{D}_1)$ and $X\in \Gamma((\ker
F_*)^{\perp})$.

(c)
\begin{align*}
  &(\nabla F_*)(V,JW)\in \Gamma(F_* \mu^J),  \\
  &\displaystyle{g_N ((\nabla F_*)(V,JW), F_* C_J X) = \lambda^2g_M (W,
\mathcal{T}_V \omega_J B_J X)}
\end{align*}
for $V,W\in \Gamma(\mathcal{D}_1)$ and $X\in \Gamma((\ker
F_*)^{\perp})$.

(d)
\begin{align*}
  &(\nabla F_*)(V,KW)\in \Gamma(F_* \mu^K),   \\
  &\displaystyle{g_N ((\nabla F_*)(V,KW), F_* C_K X) = \lambda^2g_M (W,
\mathcal{T}_V \omega_K B_K X)}
\end{align*}
for $V,W\in \Gamma(\mathcal{D}_1)$ and $X\in \Gamma((\ker
F_*)^{\perp})$.
\end{theorem}

\begin{proof}
Given $U,V\in \Gamma(\mathcal{D}_1)$, $W\in \Gamma(\mathcal{D}_2)$,
and $R\in \{ I,J,K \}$, we get
\begin{align*}
  g_M (\nabla_V U, W)
  &= g_M (\mathcal{H}\nabla_V RU, RW)  \\
  &= -\frac{1}{\lambda^2}g_N ((\nabla F_*)(V,RU), F_* RW)
\end{align*}
so that
\begin{equation}\label{eq24}
g_M (\nabla_V U, W) = 0 \Leftrightarrow (\nabla F_*)(V,RU)\in
\Gamma(F_* \mu^R).
\end{equation}
Given $X\in \Gamma((\ker F_*)^{\perp})$, we obtain
\begin{align*}
  g_M (\nabla_V U, X)
  &= g_M (U, \nabla_V RB_R X) + g_M (\mathcal{H}\nabla_V RU, C_R X)  \\
  &= g_M (U, \mathcal{T}_V \omega_R B_R X) -\frac{1}{\lambda^2}g_N ((\nabla F_*)(V,RU), F_* C_R X)
\end{align*}
so that
\begin{equation}\label{eq25}
g_M (\nabla_V U, X) = 0 \Leftrightarrow g_N ((\nabla F_*)(V,RU), F_*
C_R X) = \lambda^2g_M (U, \mathcal{T}_V \omega_R B_R X).
\end{equation}
Using (\ref{eq24}) and (\ref{eq25}), we have $(a) \Leftrightarrow
(b)$, $(a) \Leftrightarrow (c)$, $(a) \Leftrightarrow (d)$.

Therefore, we obtain the result.
\end{proof}

\begin{theorem}\label{thm5}
Let $F$ be an h-conformal semi-invariant submersion from a
hyperk\"{a}hler manifold $(M,I,J,K,g_M)$ onto a Riemannian manifold
$(N, g_N)$ such that $(I,J,K)$ is an h-conformal semi-invariant
basis. Then the following conditions are equivalent:

(a) the distribution $\mathcal{D}_2$ defines a totally geodesic
foliation on $M$.

(b) $(\nabla F_*)(V,IW)\in \Gamma(F_* \mu^I)$,
\begin{align*}
  \displaystyle{-\frac{1}{\lambda^2}g_N (\nabla_{IV}^F F_* IU, F_* IC_I X)}
  &= g_M (V, B_I \mathcal{T}_U B_I X)  \\
  &+g_M (U, V) g_M (\mathcal{H}\nabla(\ln \lambda), IC_I X)
\end{align*}
for $U,V\in \Gamma(\mathcal{D}_2)$, $W\in \Gamma(\mathcal{D}_1)$,
and $X\in \Gamma((\ker F_*)^{\perp})$.

(c) $(\nabla F_*)(V,JW)\in \Gamma(F_* \mu^J)$,
\begin{align*}
  \displaystyle{-\frac{1}{\lambda^2}g_N (\nabla_{JV}^F F_* JU, F_* JC_J X)}
  &= g_M (V, B_J \mathcal{T}_U B_J X)  \\
  &+g_M (U, V) g_M (\mathcal{H}\nabla(\ln \lambda), JC_J X)
\end{align*}
for $U,V\in \Gamma(\mathcal{D}_2)$, $W\in \Gamma(\mathcal{D}_1)$,
and $X\in \Gamma((\ker F_*)^{\perp})$.

(d) $(\nabla F_*)(V,KW)\in \Gamma(F_* \mu^K)$,
\begin{align*}
  \displaystyle{-\frac{1}{\lambda^2}g_N (\nabla_{KV}^F F_* KU, F_* KC_K X)}
  &= g_M (V, B_K \mathcal{T}_U B_K X)  \\
  &+g_M (U, V) g_M (\mathcal{H}\nabla(\ln \lambda), KC_K X)
\end{align*}
for $U,V\in \Gamma(\mathcal{D}_2)$, $W\in \Gamma(\mathcal{D}_1)$,
and $X\in \Gamma((\ker F_*)^{\perp})$.
\end{theorem}

\begin{proof}
Given $U,V\in \Gamma(\mathcal{D}_2)$, $W\in \Gamma(\mathcal{D}_1)$,
$R\in \{ I,J,K \}$, we get
$$
g_M (\nabla_U V, W) = \frac{1}{\lambda^2}g_N ((\nabla F_*)(U,RW),
F_* RV)
$$
so that
\begin{equation}\label{eq26}
g_M (\nabla_U V, W) = 0 \Leftrightarrow (\nabla F_*)(U,RW)\in
\Gamma(F_* \mu^R).
\end{equation}
Given $X\in \Gamma((\ker F_*)^{\perp})$, by using (\ref{eqn5}),
(\ref{eqn77}), (\ref{eqn92}), we obtain
\begin{align*}
  g_M (\nabla_U V, X)
  &= -g_M (RV, \mathcal{T}_U B_R X) + g_M (\nabla_{RV} U, C_R X)  \\
  &= -g_M (RV, \mathcal{T}_U B_R X) + g_M (\nabla_{RV} RU, RC_R X)  \\
  &= g_M (V, B_R\mathcal{T}_U B_R X) + g_M (U, V) g_M (\mathcal{H}\nabla(\ln \lambda), RC_R X)  \\
  &+\frac{1}{\lambda^2}g_N (\nabla_{RV}^F F_* RU, F_* RC_R X)
\end{align*}
so that
\begin{eqnarray}
  &&g_M (\nabla_U V, X) = 0  \label{eq27}  \\
  &&\Leftrightarrow -\frac{1}{\lambda^2}g_N (\nabla_{RV}^F F_* RU, F_* RC_R X) \nonumber  \\
  &&= g_M (V, B_R\mathcal{T}_U B_R X) + g_M (U, V) g_M (\mathcal{H}\nabla(\ln \lambda), RC_R X). \nonumber
\end{eqnarray}
Using (\ref{eq26}) and (\ref{eq27}), we have $(a) \Leftrightarrow
(b)$, $(a) \Leftrightarrow (c)$, $(a) \Leftrightarrow (d)$.

Therefore, the result follows.
\end{proof}

Using Theorem \ref{thm4} and Theorem \ref{thm5}, we obtain

\begin{theorem}
Let $F$ be an h-conformal semi-invariant submersion from a
hyperk\"{a}hler manifold $(M,I,J,K,g_M)$ onto a Riemannian manifold
$(N, g_N)$ such that $(I,J,K)$ is an h-conformal semi-invariant
basis. Then the following conditions are equivalent:

(a) the fibers of $F$ are locally product Riemannian manifolds
$M_{\mathcal{D}_1} \times M_{\mathcal{D}_2}$.

(b)
\begin{align*}
  &(\nabla F_*)(V,IW)\in \Gamma(F_* \mu^I),   \\
  &\displaystyle{g_N ((\nabla F_*)(V,IW), F_* C_I X) = \lambda^2g_M (W,
\mathcal{T}_V \omega_I B_I X)}
\end{align*}
for $V,W\in \Gamma(\mathcal{D}_1)$ and $X\in \Gamma((\ker
F_*)^{\perp})$.

$(\nabla F_*)(V,IW)\in \Gamma(F_* \mu^I)$,
\begin{align*}
  \displaystyle{-\frac{1}{\lambda^2}g_N (\nabla_{IV}^F F_* IU, F_* IC_I X)}
  &= g_M (V, B_I \mathcal{T}_U B_I X)  \\
  &+g_M (U, V) g_M (\mathcal{H}\nabla(\ln \lambda), IC_I X)
\end{align*}
for $U,V\in \Gamma(\mathcal{D}_2)$, $W\in \Gamma(\mathcal{D}_1)$,
and $X\in \Gamma((\ker F_*)^{\perp})$.

(c)
\begin{align*}
  &(\nabla F_*)(V,JW)\in \Gamma(F_* \mu^J),  \\
  &\displaystyle{g_N ((\nabla F_*)(V,JW), F_* C_J X) = \lambda^2g_M (W,
\mathcal{T}_V \omega_J B_J X)}
\end{align*}
for $V,W\in \Gamma(\mathcal{D}_1)$ and $X\in \Gamma((\ker
F_*)^{\perp})$.

$(\nabla F_*)(V,JW)\in \Gamma(F_* \mu^J)$,
\begin{align*}
  \displaystyle{-\frac{1}{\lambda^2}g_N (\nabla_{JV}^F F_* JU, F_* JC_J X)}
  &= g_M (V, B_J \mathcal{T}_U B_J X)  \\
  &+g_M (U, V) g_M (\mathcal{H}\nabla(\ln \lambda), JC_J X)
\end{align*}
for $U,V\in \Gamma(\mathcal{D}_2)$, $W\in \Gamma(\mathcal{D}_1)$,
and $X\in \Gamma((\ker F_*)^{\perp})$.

(d)
\begin{align*}
  &(\nabla F_*)(V,KW)\in \Gamma(F_* \mu^K),   \\
  &\displaystyle{g_N ((\nabla F_*)(V,KW), F_* C_K X) = \lambda^2g_M (W,
\mathcal{T}_V \omega_K B_K X)}
\end{align*}
for $V,W\in \Gamma(\mathcal{D}_1)$ and $X\in \Gamma((\ker
F_*)^{\perp})$.

$(\nabla F_*)(V,KW)\in \Gamma(F_* \mu^K)$,
\begin{align*}
  \displaystyle{-\frac{1}{\lambda^2}g_N (\nabla_{KV}^F F_* KU, F_* KC_K X)}
  &= g_M (V, B_K \mathcal{T}_U B_K X)  \\
  &+g_M (U, V) g_M (\mathcal{H}\nabla(\ln \lambda), KC_K X)
\end{align*}
for $U,V\in \Gamma(\mathcal{D}_2)$, $W\in \Gamma(\mathcal{D}_1)$,
and $X\in \Gamma((\ker F_*)^{\perp})$.
\end{theorem}

We know

\begin{lemma}\cite{BW}\label{lem3}
Let $F$ be a horizontally conformal submersion from a Riemannian
manifold $(M,g_M)$ onto a Riemannian manifold $(N,g_N)$ with
dilation $\lambda$.

Then the tension field $\tau(F)$ of $F$ is given by
\begin{equation}\label{eq28}
\tau(F) = -mF_* H + (2-n)F_* (\nabla(\ln \lambda)),
\end{equation}
where $H$ is the mean curvature vector field of the distribution
$\ker F_*$, $m = \dim \ker F_*$, $n = \dim N$.
\end{lemma}

Using Lemma \ref{lem3}, we easily get

\begin{corollary}
Let $F$ be an almost h-conformal semi-invariant submersion from a
hyperk\"{a}hler manifold $(M,I,J,K,g_M)$ onto a Riemannian manifold
$(N, g_N)$ such that $(I,J,K)$ is an almost h-conformal
semi-invariant basis. Assume that $F$ is harmonic with $\dim \ker
F_* > 0$ and $\dim N > 2$. Then the following conditions are
equivalent:

(a) all the fibers of $F$ are minimal.

(b) the map $F$ is horizontally homothetic.
\end{corollary}

\begin{corollary}
Let $F$ be an almost h-conformal semi-invariant submersion from a
hyperk\"{a}hler manifold $(M,I,J,K,g_M)$ onto a Riemannian manifold
$(N, g_N)$ such that $(I,J,K)$ is an almost h-conformal
semi-invariant basis. Assume that $\dim \ker F_* > 0$ and $\dim N =
2$. Then the following conditions are equivalent:

(a) all the fibers of $F$ are minimal.

(b) the map $F$ is harmonic.
\end{corollary}

We introduce another notion and investigate the condition for such a
map to be totally geodesic.

\begin{definition}
Let $F$ be an almost h-conformal semi-invariant submersion from a
hyperk\"{a}hler manifold $(M,I,J,K,g_M)$ onto a Riemannian manifold
$(N, g_N)$ such that $(I,J,K)$ is an almost h-conformal
semi-invariant basis. Then given $R\in \{ I,J,K \}$, we call the map
$F$ a {\em $(R\mathcal{D}_2^R, \mu^R)$-totally geodesic map} if
$(\nabla F_*)(RU,X) = 0$ for $U\in \Gamma(\mathcal{D}_2^R)$ and
$X\in \Gamma(\mu^R)$.
\end{definition}

\begin{theorem}
Let $F$ be an almost h-conformal semi-invariant submersion from a
hyperk\"{a}hler manifold $(M,I,J,K,g_M)$ onto a Riemannian manifold
$(N, g_N)$ such that $(I,J,K)$ is an almost h-conformal
semi-invariant basis. Then the following conditions are equivalent:

(a) the map $F$ is horizontally homothetic.

(b) the map $F$ is a $(I\mathcal{D}_2^I, \mu^I)$-totally geodesic
map.

(c) the map $F$ is a $(J\mathcal{D}_2^J, \mu^J)$-totally geodesic
map.

(d) the map $F$ is a $(K\mathcal{D}_2^K, \mu^K)$-totally geodesic
map.
\end{theorem}

\begin{proof}
Given $U\in \Gamma(\mathcal{D}_2^R)$, $X\in \Gamma(\mu^R)$, and
$R\in \{ I,J,K \}$, we have
\begin{eqnarray}
  &&(\nabla F_*)(RU,X)  \label{eq29}   \\
  &&= RU(\ln \lambda)F_* X + X(\ln \lambda)F_* RU - g_M (RU, X)F_* \nabla(\ln \lambda)  \nonumber   \\
  &&= RU(\ln \lambda)F_* X + X(\ln \lambda)F_* RU   \nonumber
\end{eqnarray}
so that we easily get $(a) \Rightarrow (b)$, $(a) \Rightarrow (c)$,
$(a) \Rightarrow (d)$.

Conversely, from (\ref{eq29}), we obtain
$$
RU(\ln \lambda)F_* X + X(\ln \lambda)F_* RU = 0.
$$
Since $\{ F_* X, F_* RU \}$ is linearly independent for nonzero $X$,
$U$, we have $RU(\ln \lambda) = 0$ and $X(\ln \lambda) = 0$, which
means $(a) \Leftarrow (b)$, $(a) \Leftarrow (c)$, $(a) \Leftarrow
(d)$.

Therefore, the result follows.
\end{proof}

\begin{theorem}
Let $F$ be an almost h-conformal semi-invariant submersion from a
hyperk\"{a}hler manifold $(M,I,J,K,g_M)$ onto a Riemannian manifold
$(N, g_N)$ such that $(I,J,K)$ is an almost h-conformal
semi-invariant basis. Then the following conditions are equivalent:

(a) the map $F$ is totally geodesic.

(b) (i) $C_I \mathcal{T}_U IV + \omega_I \widehat{\nabla}_U IV = 0$
for $U,V\in \Gamma(\mathcal{D}_1^I)$.

(ii) $C_I \mathcal{H}\nabla_U IW + \omega_I \mathcal{T}_U IW = 0$
for $U\in \Gamma(\ker F_*)$ and $W\in \Gamma(\mathcal{D}_2^I)$.

(iii) the map $F$ is horizontally homothetic.

(iv) $\mathcal{T}_U B_I X + \mathcal{H}\nabla_U C_I X\in
\Gamma(I\mathcal{D}_2^I)$ and $\widehat{\nabla}_U B_I X +
\mathcal{T}_U C_I X\in \Gamma(\mathcal{D}_1^I)$ for $U\in
\Gamma(\ker F_*)$ and $X\in \Gamma((\ker F_*)^{\perp})$.

(c) (i) $C_J \mathcal{T}_U JV + \omega_J \widehat{\nabla}_U JV = 0$
for $U,V\in \Gamma(\mathcal{D}_1^J)$.

(ii) $C_J \mathcal{H}\nabla_U JW + \omega_J \mathcal{T}_U JW = 0$
for $U\in \Gamma(\ker F_*)$ and $W\in \Gamma(\mathcal{D}_2^J)$.

(iii) the map $F$ is horizontally homothetic.

(iv) $\mathcal{T}_U B_J X + \mathcal{H}\nabla_U C_J X\in
\Gamma(J\mathcal{D}_2^J)$ and $\widehat{\nabla}_U B_J X +
\mathcal{T}_U C_J X\in \Gamma(\mathcal{D}_1^J)$ for $U\in
\Gamma(\ker F_*)$ and $X\in \Gamma((\ker F_*)^{\perp})$.

(d) (i) $C_K \mathcal{T}_U KV + \omega_K \widehat{\nabla}_U KV = 0$
for $U,V\in \Gamma(\mathcal{D}_1^K)$.

(ii) $C_K \mathcal{H}\nabla_U KW + \omega_K \mathcal{T}_U KW = 0$
for $U\in \Gamma(\ker F_*)$ and $W\in \Gamma(\mathcal{D}_2^K)$.

(iii) the map $F$ is horizontally homothetic.

(iv) $\mathcal{T}_U B_K X + \mathcal{H}\nabla_U C_K X\in
\Gamma(K\mathcal{D}_2^K)$ and $\widehat{\nabla}_U B_K X +
\mathcal{T}_U C_K X\in \Gamma(\mathcal{D}_1^K)$ for $U\in
\Gamma(\ker F_*)$ and $X\in \Gamma((\ker F_*)^{\perp})$.
\end{theorem}

\begin{proof}
Given $U,V\in \Gamma(\mathcal{D}_1^R)$ and $R\in \{ I,J,K \}$, we
have
\begin{align*}
  (\nabla F_*)(U,V)
  &= F_* (R(\mathcal{T}_U RV + \widehat{\nabla}_U RV))  \\
  &= F_* (B_R \mathcal{T}_U RV + C_R \mathcal{T}_U RV + \phi_R \widehat{\nabla}_U RV + \omega_R \widehat{\nabla}_U RV)  \\
  &= F_* (C_R \mathcal{T}_U RV + \omega_R \widehat{\nabla}_U RV)
\end{align*}
so that
\begin{equation}\label{eq30}
(\nabla F_*)(U,V) = 0 \Leftrightarrow C_R \mathcal{T}_U RV +
\omega_R \widehat{\nabla}_U RV = 0.
\end{equation}
Given $U\in \Gamma(\ker F_*)$ and $W\in \Gamma(\mathcal{D}_2^R)$, we
get
\begin{align*}
  (\nabla F_*)(U,W)
  &= F_* (R(\nabla_U RW))   \\
  &= F_* (R(\mathcal{T}_U RW + \mathcal{H}\nabla_U RW))  \\
  &= F_* (C_R \mathcal{H}\nabla_U RW + \omega_R \mathcal{T}_U RW)
\end{align*}
so that
\begin{equation}\label{eq31}
(\nabla F_*)(U,W) = 0 \Leftrightarrow C_R \mathcal{H}\nabla_U RW +
\omega_R \mathcal{T}_U RW = 0.
\end{equation}
We claim that
\begin{eqnarray}
  &&(\nabla F_*)(X,Y) = 0 \quad \text{for} \ X,Y\in \Gamma((\ker F_*)^{\perp})   \label{eq32}   \\
  &&\Leftrightarrow F \ \text{is horizontally homothetic.}   \nonumber
\end{eqnarray}
Given $X,Y\in \Gamma((\ker F_*)^{\perp})$, by (\ref{eqn77}),  we
obtain
\begin{equation}\label{eq33}
(\nabla F_*)(X,Y) = X(\ln \lambda)F_* Y + Y(\ln \lambda)F_* X - g_M
(X,Y)F_* \nabla(\ln \lambda)
\end{equation}
so that the part from right to left  immediately follows.

Conversely, we have
\begin{equation}\label{eq34}
0 = X(\ln \lambda)F_* Y + Y(\ln \lambda)F_* X - g_M (X,Y)F_*
\nabla(\ln \lambda).
\end{equation}
Applying $Y = RX$, $X\in \Gamma(\mu^R)$, at (\ref{eq34}), we get
\begin{align*}
  0
  &= X(\ln \lambda)F_* RX + RX(\ln \lambda)F_* X - g_M (X,RX)F_*
\nabla(\ln \lambda)   \\
  &= X(\ln \lambda)F_* RX + RX(\ln \lambda)F_* X
\end{align*}
so that since $\{ F_* RX, F_* X \}$ is linearly independent for
nonzero $X$, we obtain $X(\ln \lambda) = 0$ and $RX(\ln \lambda) =
0$, which implies
\begin{equation}\label{eq35}
X(\lambda) = 0 \quad \text{for} \ X\in \Gamma(\mu^R).
\end{equation}
Applying $X = Y = RU$, $U\in \Gamma(\mathcal{D}_2^R)$, at
(\ref{eq34}), we obtain
\begin{equation}\label{eq36}
0 = 2RU(\ln \lambda)F_* RU - g_M (RU,RU)F_* \nabla(\ln \lambda).
\end{equation}
Taking inner product with $F_* RU$ at (\ref{eq36}), we have
\begin{align*}
  0
  &= 2g_M (RU, \nabla(\ln \lambda)) g_N (F_* RU, F_* RU) - g_M (RU, RU) g_N (F_* \nabla(\ln \lambda), F_* RU)   \\
  &= \lambda g_M (RU, RU) g_M (RU, \nabla(\ln \lambda)),
\end{align*}
which implies
\begin{equation}\label{eq37}
RU(\lambda) = 0 \quad \text{for} \ U\in \Gamma(\mathcal{D}_2^R).
\end{equation}
By (\ref{eq35}) and (\ref{eq37}), we get the part from left to
right.

Given $U\in \Gamma(\ker F_*)$ and $X\in \Gamma((\ker F_*)^{\perp})$,
we obtain
\begin{align*}
  (\nabla F_*)(U,X)
  &= F_* (R(\nabla_U RX))   \\
  &= F_* (R(\mathcal{T}_U B_R X + \widehat{\nabla}_U B_R X) + R(\mathcal{T}_U C_R X + \mathcal{H}\nabla_U C_R X))  \\
  &= F_* (C_R (\mathcal{T}_U B_R X + \mathcal{H}\nabla_U C_R X) + \omega_R (\widehat{\nabla}_U B_R X + \mathcal{T}_U C_R X))
\end{align*}
so that
\begin{eqnarray}
(\nabla F_*)(U,X) = 0 \Leftrightarrow
  &&\mathcal{T}_U B_R X + \mathcal{H}\nabla_U C_R X \in \Gamma(R\mathcal{D}_2^R),    \label{eq38}   \\
  &&\widehat{\nabla}_U B_R X + \mathcal{T}_U C_R X \in \Gamma(R\mathcal{D}_1^R)  \nonumber
\end{eqnarray}
By (\ref{eq30}), (\ref{eq31}), (\ref{eq32}), (\ref{eq38}), we have
$(a) \Leftrightarrow (b)$, $(a) \Leftrightarrow (c)$, $(a)
\Leftrightarrow (d)$.

Therefore, we get the result.
\end{proof}

Let $F : (M,g_M)\mapsto (N,g_N)$ be a horizontally conformal
submersion. The map $F$ is called a horizontally conformal
submersion {\em with totally umbilical fibers} if
\begin{equation}\label{eq39}
\mathcal{T}_X Y=g_M (X, Y)H \quad  \text{for} \  X,Y\in \Gamma(\ker
F_*),
\end{equation}
where $H$ is the mean curvature vector field of the distribution
$\ker F_*$.

\begin{lemma}\label{lem4}
Let $F$ be an almost h-conformal semi-invariant submersion with
totally umbilical fibers from a hyperk\"{a}hler manifold
$(M,I,J,K,g_M)$ onto a Riemannian manifold $(N, g_N)$ such that
$(I,J,K)$ is an almost h-conformal semi-invariant basis. Then
\begin{equation}\label{eq40}
H\in \Gamma(R\mathcal{D}_2^R) \quad  \text{for} \ R\in \{ I,J,K \}.
\end{equation}
\end{lemma}

\begin{proof}
Given $X, Y\in \Gamma (\mathcal{D}_1^R)$, $W\in \Gamma (\mu^R)$, and
$R\in \{ I,J,K \}$, we have
\begin{align*}
&\mathcal{T}_X RY+\widehat{\nabla}_X RY= \nabla_X RY= R\nabla_X Y \\
&=B_R \mathcal{T}_X Y+C_R \mathcal{T}_X Y+\phi_R
\widehat{\nabla}_X Y+\omega_R \widehat{\nabla}_X Y  \\
\end{align*}
so that
$$
g_M (\mathcal{T}_X RY, W)= g_M (C_R \mathcal{T}_X Y, W) = -g_M
(\mathcal{T}_X Y, RW).
$$
Using (\ref{eq39}),  we obtain
$$
g_M (X, RY) g_M (H, W)= -g_M (X, Y) g_M (H, RW).
$$
Interchanging the role of $X$ and $Y$, we get
$$
g_M (Y, RX) g_M (H, W)= -g_M (Y, X) g_M (H, RW).
$$
Combining the above two equations, we have
$$
g_M (X, Y) g_M (H, RW)=0,
$$
which implies $H\in \Gamma(R\mathcal{D}_2^R)$ (since $R\mu^R =
\mu^R$).
\end{proof}

\begin{theorem}
Let $F$ be an h-conformal semi-invariant submersion with totally
umbilical fibers from a hyperk\"{a}hler manifold $(M,I,J,K,g_M)$
onto a Riemannian manifold $(N, g_N)$ such that $(I,J,K)$ is an
 h-conformal semi-invariant basis. Then all the fibers of $F$
are totally geodesic.
\end{theorem}

\begin{proof}
By Lemma \ref{lem4}, we have
$$
H\in \Gamma(R\mathcal{D}_2) \quad \text{for} \ R\in \{ I,J,K \}
$$
so that
$$
\{ IH, JH, KH \} \ \subset \ \Gamma(\mathcal{D}_2).
$$
But
$$
KH = IJH = I(JH)\in \Gamma(\mathcal{D}_2) \quad \text{with} \ JH\in
\Gamma(\mathcal{D}_2).
$$
Since $I\mathcal{D}_2 \subset (\ker F_*)^{\perp}$, we must have $H =
0$. By (\ref{eq39}), we obtain the result.
\end{proof}

\section{Examples}

Note that given an Euclidean space $\mathbb{R}^{4m}$ with
coordinates $(x_1,x_2,\cdots,x_{4m})$, we can canonically choose
complex structures $I, J, K$ on $\mathbb{R}^{4m}$ as follows:
\begin{align*}
  &I(\tfrac{\partial}{\partial x_{4k+1}})=\tfrac{\partial}{\partial x_{4k+2}},
  I(\tfrac{\partial}{\partial x_{4k+2}})=-\tfrac{\partial}{\partial x_{4k+1}},
  I(\tfrac{\partial}{\partial x_{4k+3}})=\tfrac{\partial}{\partial x_{4k+4}},
  I(\tfrac{\partial}{\partial x_{4k+4}})=-\tfrac{\partial}{\partial x_{4k+3}},     \\
  &J(\tfrac{\partial}{\partial x_{4k+1}})=\tfrac{\partial}{\partial x_{4k+3}},
  J(\tfrac{\partial}{\partial x_{4k+2}})=-\tfrac{\partial}{\partial x_{4k+4}},
  J(\tfrac{\partial}{\partial x_{4k+3}})=-\tfrac{\partial}{\partial x_{4k+1}},
  J(\tfrac{\partial}{\partial x_{4k+4}})=\tfrac{\partial}{\partial x_{4k+2}},    \\
  &K(\tfrac{\partial}{\partial x_{4k+1}})=\tfrac{\partial}{\partial x_{4k+4}},
  K(\tfrac{\partial}{\partial x_{4k+2}})=\tfrac{\partial}{\partial x_{4k+3}},
  K(\tfrac{\partial}{\partial x_{4k+3}})=-\tfrac{\partial}{\partial x_{4k+2}},
  K(\tfrac{\partial}{\partial x_{4k+4}})=-\tfrac{\partial}{\partial x_{4k+1}}
\end{align*}
for $k\in \{ 0,1,\cdots,m-1 \}$.

Then we easily check that $(I,J,K,\langle \ ,\ \rangle)$ is a
hyperk\"{a}hler structure on $\mathbb{R}^{4m}$, where $\langle \ ,\
\rangle$ denotes the Euclidean metric on $\mathbb{R}^{4m}$.
Throughout this section, we will use these notations.

\begin{example}
Let $(M, E, g)$ be an almost quaternionic Hermitian manifold. Let
$\pi : TM \mapsto M$ be the natural projection \cite{IMV}. Then the
map $\pi$ is an h-conformal semi-invariant submersion such that
$\mathcal{D}_1=\ker \pi_*$ and dilation $\lambda = 1$.
\end{example}

\begin{example}
Let $(M,E_M,g_M)$ and $(N,E_N,g_N)$ be almost quaternionic Hermitian
manifolds. Let $F : M \mapsto N$ be a quaternionic submersion
\cite{IMV}. Then the map $F$ is an h-conformal semi-invariant
submersion such that
 $\mathcal{D}_1= \ker F_*$ and dilation $\lambda = 1$.
\end{example}

\begin{example}
Let $(M,E,g_M)$ be an almost quaternionic Hermitian manifold and
$(N,g_N)$ a Riemannian manifold. Let $F : (M,E,g_M) \mapsto (N,g_N)$
be an h-semi-invariant submersion \cite{P2}. Then the map $F$ is an
h-conformal semi-invariant submersion with dilation $\lambda = 1$.
\end{example}

\begin{example}
Let $(M,E,g_M)$ be an almost quaternionic Hermitian manifold and
$(N,g_N)$ a Riemannian manifold. Let $F : (M,E,g_M) \mapsto (N,g_N)$
be an almost h-semi-invariant submersion \cite{P2}. Then the map $F$
is an almost h-conformal semi-invariant submersion with dilation
$\lambda = 1$.
\end{example}

\begin{example}
Let $(M,E,g_M)$ be a $4n$-dimensional almost quaternionic Hermitian
manifold and $(N,g_N)$ a $(4n-1)$-dimensional Riemannian manifold.
Let $F : (M,E,g_M) \mapsto (N,g_N)$ be a horizontally conformal
submersion with dilation $\lambda$. Then the map $F$ is an
h-conformal semi-invariant submersion such that $\mathcal{D}_2 =
\ker F_*$ and dilation $\lambda$.
\end{example}

\begin{example}
Let $F : \mathbb{R}^4 \mapsto \mathbb{R}^3$ be a horizontally
conformal submersion with dilation $\lambda$. Then the map $F$ is an
h-conformal semi-invariant submersion such that $\mathcal{D}_2= \ker
F_*$ and dilation $\lambda$.
\end{example}

\begin{example}\label{exam1}
Define a map $F : \mathbb{R}^4 \mapsto \mathbb{R}^2$ by
$$
F(x_1,\cdots,x_4) = e^{1934}(x_1, x_2).
$$
Then the map $F$ is an almost h-conformal anti-holomorphic
semi-invariant submersion such that $I(\ker F_*)=\ker F_*$, $J(\ker
F_*)=(\ker F_*)^\perp$,
 $K(\ker F_*)=(\ker F_*)^\perp$, and dilation $\lambda = e^{1934}$.

Here, $(K,I,J)$ is an almost h-conformal anti-holomorphic
semi-invariant basis.
\end{example}

\begin{example}
Define a map $F : \mathbb{R}^8 \mapsto \mathbb{R}^6$ by
$$
F(x_1,\cdots,x_8) = \pi^{1934}(x_3,\cdots,x_8).
$$
Then the map $F$ is an almost h-conformal semi-invariant submersion
such that $I(\ker F_*)=\ker F_*$, $J(\ker F_*)\subset(\ker
F_*)^\perp$,  $K(\ker F_*)\subset(\ker F_*)^\perp$, and dilation
$\lambda = \pi^{1934}$.
\end{example}

\begin{example}
Define a map $F : \mathbb{R}^8 \mapsto \mathbb{R}^4$ by
$$
F(x_1,\cdots,x_8) = e^{1968}(x_1,x_2,x_5,x_7).
$$
Then the map $F$ is an almost h-conformal semi-invariant submersion
such that $\mathcal{D}_1^I =\mathcal{D}_2^J =
<\frac{\partial}{\partial x_3}, \frac{\partial}{\partial x_4}>$,
$\mathcal{D}_2^I =\mathcal{D}_1^J = <\frac{\partial}{\partial x_6},
\frac{\partial}{\partial x_8}>$,  $K(\ker F_*)=(\ker F_*)^\perp$,
and dilation $\lambda = e^{1968}$.
\end{example}

\begin{example}
Define a map $F : \mathbb{R}^8 \mapsto \mathbb{R}^3$ by
$$
F(x_1,\cdots,x_8) = \pi^{1968}(x_6,x_7,x_8).
$$
Then the map $F$ is a h-conformal semi-invariant submersion  such
that $\mathcal{D}_1 = <\frac{\partial}{\partial x_1}, \cdots,
\frac{\partial}{\partial x_4}>$,  $\mathcal{D}_2 =
<\frac{\partial}{\partial x_5}>$, and dilation $\lambda =
\pi^{1968}$.
\end{example}



\end{document}